\documentclass{llncs}

\usepackage{microtype,comment}
\usepackage{xypic}
\usepackage[T1]{fontenc}

\title{Few Induced Disjoint Paths for $H$-Free Graphs\thanks{An extended abstract of this paper will appear in the proceedings of ISCO 2022~\cite{MPSV2}.}}
\author{Barnaby Martin\inst{1}\and Dani\"el Paulusma\inst{1}  \and Siani Smith\inst{1} \and Erik Jan van Leeuwen\inst{2}}

\institute{Department of Computer Science,
 Durham University, Durham, UK,
\email{\{barnaby.d.martin,daniel.paulusma,siani.smith\}@durham.ac.uk}
\and
Department of Information and Computing Sciences, 
Utrecht University,\\
The Netherlands,
\email{e.j.vanleeuwen@uu.nl}
}

\usepackage[ruled,vlined,linesnumbered,boxed]{algorithm2e}

\usepackage{todonotes}

\newcommand{\ssi}{\subseteq_i}

\usepackage{enumerate}
\usepackage{tikz}
\usepackage{boxedminipage,amsmath}
\usetikzlibrary{arrows,shapes,calc}

\newcommand{\problemdef}[3]{
	\begin{center}
		\begin{boxedminipage}{1.02\textwidth}
			\textsc{{#1}}\\[1pt]  
			\begin{tabular}{ r p{0.8\textwidth}}
				\textit{~~~~Instance:} & {#2}\\
				\textit{Question:} & {#3}
			\end{tabular}
		\end{boxedminipage}
	\end{center}
}

\newcommand{\NP}{{\sf NP}}

\newcommand{\FPT}{{\sf FPT}}
\newcommand{\XP}{{\sf XP}}
\newcommand{\W}{{\sf W}}

\newcounter{ctrclaim}[theorem]
\newcounter{ctrcase}[theorem]

\usepackage{tikz}
\usetikzlibrary{arrows}
\usetikzlibrary{shapes}

\usetikzlibrary{shapes}
\newtheorem{claimm}{Claim}

\begin{document}

\maketitle

\begin{abstract}
Paths $P^1,\ldots,P^k$ in a graph $G=(V,E)$ are mutually induced if any two distinct $P^i$ and $P^j$ have neither common vertices nor adjacent vertices. For a fixed integer~$k$, the {\sc $k$-Induced Disjoint Paths} problem is to decide if a graph~$G$ with $k$ pairs of specified vertices $(s_i,t_i)$ contains $k$ mutually induced paths~$P^i$ such that each $P^i$ starts from $s_i$ and ends at~$t_i$. 
Whereas the non-induced version is well-known to be polynomial-time solvable for every fixed integer $k$, a classical result from the literature states that even {\sc $2$-Induced Disjoint Paths} is \NP-complete.
We prove new complexity results for {\sc $k$-Induced Disjoint Paths} if the input is restricted to $H$-free graphs, that is, graphs without a fixed graph $H$ as an induced subgraph. We compare our results with a complexity dichotomy for {\sc Induced Disjoint Paths}, the variant where $k$ is part of the input.
\end{abstract}

\section{Introduction}\label{s-intro}

We consider problems related to finding paths connecting pre-specified pairs of vertices. 
A path between vertices $s$ and $t$ in an undirected graph~$G$ is an {\it $s$-$t$ path} with {\it terminals} $s$ and $t$.
Terminal pairs $(s_1,t_1),\ldots, (s_k,t_k)$ are {\it pairwise disjoint} if $\{s_i,t_i\}\cap \{s_j,t_j\}=\emptyset$ for $i\neq j$. The well-known problem {\sc $k$-Disjoint Paths} is to decide for a graph $G$ and pairwise~disjoint~terminal~pairs~$(s_1,t_1)\ldots,(s_k,t_k)$, if there are pairwise vertex-disjoint paths $P^1$,\ldots,$P^k$ such that $P^i$ is an $(s_i,t_i)$-path for $i\in \{1,\ldots,k\}$;
here $k$ is {\it fixed}, that is, $k$ is not part of the input.
 
Shiloach~\cite{Sh80} proved that $2$-{\sc Disjoint Paths} is polynomial-time solvable.
Robertson and Seymour~\cite{RS95} even gave a polynomial-time algorithm for {\sc $k$-Disjoint Paths} for every integer~$k\geq 2$. In contrast, {\sc Disjoint Paths}, the variant where $k$ is part of the input, appeared on Karp's list of \NP-complete problems.

\medskip
\noindent
{\bf Our Focus.} We consider the {\it induced} variant of {\sc $k$-Disjoint Paths}. We say that paths $P^1,\ldots,P^k$ in a graph $G=(V,E)$ are {\it mutually induced} if any two distinct $P^i$ and $P^j$ have neither common vertices nor adjacent vertices, that is, if $i\neq j$ then $V(P^i)\cap V(P^j)=\emptyset$ and
$uv\notin E$ for every $u\in V(P^i)$ and $v\in V(P^j)$. This leads to the following problem, where $k$ is a fixed constant.

\problemdef{$k$-Induced Disjoint Paths}{a graph $G$ and pairwise~disjoint~terminal~pairs~$(s_1,t_1)\ldots,(s_k,t_k)$.}{Does $G$ have mutually induced paths $P^1$,\ldots,$P^k$ such that $P^i$ is an $s_i$-$t_i$ path for $i\in \{1,\ldots,k\}$?}

\noindent
In contrast to the previous setting, even {\sc $2$-Induced Disjoint Paths} is \NP-complete, as shown both by Bienstock~\cite{Bi91} and Fellows~\cite{Fe89}. Restricting the input to some special graph class might help improve our understanding of the hardness of the problem. 
To do this systematically we focus on hereditary graph classes. 

A class of graphs is {\it hereditary} if it is closed under vertex deletion. This is a natural property and non-surprisingly hereditary graph classes provide a framework that captures many well-known graph classes. In particular, it is not difficult to see that a graph class ${\cal G}$ is hereditary if and only if it can be characterized by a (unique) set ${\cal F}_{\cal G}$ of forbidden induced subgraphs. For example, if ${\cal G}$ is the class of bipartite graphs, then ${\cal F}_{\cal G}$ is the set of all odd cycles. 

The characterization by ${\cal F}_{\cal G}$ allows for a systematic study, which usually starts with the case where ${\cal F}_{\cal G}$ has size~$1$, say ${\cal F}_{\cal G}=\{H\}$ for some graph $H$. A graph is {\it $H$-free} if it cannot be modified to $H$ by a sequence of vertex deletions, and if ${\cal F}_{\cal G}=\{H\}$ we obtain the class of {\it $H$-free graphs}, which we consider in our paper.

\subsection{Related Work} 

We first discuss existing results for {\sc Induced Disjoint Paths} (where $k$ is part of the input). 
All the positive results hold for a slightly more general problem definition (see Section~\ref{s-con}). 
Golovach et al.~\cite{GPV16,GPV22} proved that 
{\sc Induced Disjoint Paths} is linear-time solvable for circular-arc graphs and polynomial-time solvable for AT-free graphs, respectively.
Belmonte et al.~\cite{BGHHKP14} showed the latter for chordal graphs, and Jaffke et al.~\cite{JKT20} did so for any graph class of bounded mim-width. 
In contrast, {\sc Induced Disjoint Paths} stays \NP-complete even for claw-free graphs~\cite{FKLP12},  line graphs of triangle-free chordless graphs~\cite{RTV21} and thus for (theta,wheel)-free graphs, and for planar graphs; to prove the latter, use a result of Lynch~\cite{Ly75} (see~\cite{GPV22}). 

The following recent dichotomy is immediately relevant for our paper. 
Let $G_1+G_2$ be the disjoint union of~two vertex-disjoint graphs $G_1$ and $G_2$, and let $sG$ denote the disjoint union of $s$ copies of a graph $G$. 
We write $F\ssi G$ if $F$ is an {\it induced} subgraph of a graph $G$, that is, $F$ can be obtained from $G$ by a sequence of vertex deletions. We let $P_r$ denote the path on $r$ vertices. A {\it linear forest} is the disjoint union of one or more paths. 

\begin{theorem}[\cite{MPSV}]\label{thm:IDPnew}
For a graph $H$, {\sc Induced Disjoint Paths} on $H$-free graphs is polynomial-time solvable if  
$H \ssi sP_3+P_6$ for some $s\geq 0$; \NP-complete if $H$ is not a linear forest; and quasipolynomial-time solvable otherwise.
\end{theorem}

\noindent
We return to Theorem~\ref{thm:IDPnew} later, and we now fix $k$. Radovanovi\'c et al.~\cite{RTV21} proved that 
 {\sc $k$-Induced Disjoint Paths} is polynomial-time solvable for (theta,wheel)-free graphs. Fiala et al.~\cite{FKLP12} proved the same result for claw-free graphs. Note that both results complement the aforementioned hardness results when $k$ is part of the input. Golovach et al.~\cite{GPV15} showed
 that {\sc Induced Disjoint Paths} is even \FPT\ with parameter~$k$ for claw-free graphs. The same holds for planar graphs~\cite{KK12}, and  even for graph classes of bounded genus, as shown by Kobayashi and Kawarabayashi~\cite{KK09}.
 Let $C_r$ denote the $r$-vertex cycle.
It follows (using Lemma~\ref{c-icred}) from a result of Leveque et al.~\cite{LLMT09} that {\sc $2$-Induced Disjoint Paths} is \NP-complete for $H$-free graphs if $H=C_r$ for every $r\geq 3$ with $r\neq 6$.
 
The generalization from paths to connected subgraphs joining sets of terminals instead of pairs has also been considered, but these results do not impact upon our work in this paper; we refer to~\cite{MPSV} for further details. Moreover, the restriction to $H$-free graphs has also been studied for {\sc Disjoint Paths} (recall that if $k$ is fixed this problem is polynomial in general~\cite{RS95}); see~\cite{KMPSV22} for a complexity classification of {\sc Disjoint Paths} for $H$-free graphs, subject to a set of three unknown cases.

\subsection{Our Results.} 

To explain our results we first introduce some extra terminology. For $r\geq 1$, the graph $K_{1,r}$ is the $(r+1)$-vertex {\it star}, i.e., the graph with vertices $x,y_1,\ldots,y_r$ and edges $xy_i$ for $i=1,\ldots,r$. The graph~$K_{1,3}$ is  known as the {\it claw}.  The {\it subdivision} of an edge $uw$ removes $uw$ and replaces~it with a new vertex $v$ and edges $uv$, $vw$.  A {\it subdivided claw} is a tree with one~vertex~$x$ of degree~$3$ and exactly three leaves. For $1\leq h\leq i\leq j$, let $S_{h,i,j}$ be the subdivided claw whose three leaves are of distance~$h$,~$i$ and~$j$ from the vertex of degree~$3$. Note that $S_{1,1,1}=K_{1,3}$. The graph $S_{1,1,2}$ is called the {\it chair} (or {\it fork}). Let ${\cal S}$ be the set of graphs, each connected component of which is a path or a subdivided claw.

Using the above terminology we can now present our main theorem.

\begin{theorem}\label{thm:k-IDP}
Let $k\geq 2$.
For a graph $H$, {\sc $k$-Induced Disjoint Paths} is polynomial-time solvable if $H$ is a subgraph of the disjoint union of a linear forest and a chair, and it is \NP-complete if $H$ is not in ${\cal S}$.
\end{theorem}

\noindent
Comparing Theorems~\ref{thm:IDPnew} and~\ref{thm:k-IDP} shows that the problem becomes tractable for an infinite family of graphs~$H$ after fixing~$k$.  As the class of claw-free graphs is contained in the class of chair-free graphs, Theorem~\ref{thm:k-IDP} extends the aforementioned polynomial-time result of Fiala et al.~\cite{FKLP12} for claw-free graphs. Moreover, the case $H=C_6$ (the $6$-vertex cycle) fills a gap in the aforementioned result of Leveque et al.~\cite{LLMT09}. As we shall explain in Section~\ref{s-np}, the \NP-hardness construction relies on their gadget but also requires significant additional work. Before doing this we first prove the polynomial-time part of Theorem~\ref{thm:k-IDP} in Section~\ref{s-poly}.
Then, in Section~\ref{s-mainmain}, we prove Theorem~\ref{thm:k-IDP}.

In Section~\ref{s-pc} we consider the problem from a parameterized complexity viewpoint. Recall that Golovach et al.~\cite{GPV15} proved
 that {\sc Induced Disjoint Paths} is  \FPT\ for claw-free graphs when parameterized by the number $k$ of paths. We consider the class of $P_r$-free graphs. 
This gives us another natural parameter, namely~$r$. However, we show by adapting a construction of Haas and Hoffmann~\cite{HH06} that even {\sc $2$-Induced Disjoint Paths} is $\W[1]$-hard for $P_r$-free graphs when parameterized by~$r$.

In Section~\ref{s-con} we summarize our findings and give a number of relevant open problems. In particular we discuss some open problems on the parameterized complexity of {\sc Induced Disjoint Paths}.

\section{Polynomial-Time Algorithms}\label{s-poly}

In this section we prove the polynomial-time part of Theorem~\ref{thm:k-IDP}.
We first show the following general result that we will need as a lemma.

\begin{lemma}\label{l-sp1}
For every linear forest $F$, if the {\sc $k$-Induced Disjoint Paths} problem is polynomial-time solvable for $H$-free graphs for some graph~$H$, then it is so for $(F+H)$-free graphs.
\end{lemma}

\begin{proof}
Let $H$ be a graph such that {\sc $k$-Induced Disjoint Paths} is polynomial-time solvable for $H$-free graphs. Let $(G,T)$ be an instance of {\sc $k$-Induced Disjoint Paths}, where $G$ is an $(F+H)$-free graph on $n$ vertices and $T$ is a set of $k$ terminal pairs $(s_i,t_i)$. 

Let $r=2|V(F)|-1$. 
Note that $F$ is an induced subgraph of $P_r$.
We check in $O(n^{k(r+1)})$ time (by  brute force) if there exists a solution $(P^1,\ldots,P^k)$ for $(G,T)$ in which each path has at most $r+1$ vertices. As $k$ and $r$ are constants, this takes polynomial time.

Suppose we have not found a solution yet. Then if a solution $(P^1,\ldots,P^k)$ exists, at least one of the paths $P^i$ in it has $r+2$ or more vertices. 
We guess which path $P^i$ will have length at least $r$. This leads to $k$ branches. We guess the first $r+1$ vertices $u_1,\ldots,u_{r+1}$ on $P^i$ after $s_i=u_0$. This leads to $O(n^r)$ further branches. We remove $s_i,u_1,\ldots,u_r$ and all their neighbours from $G$, except for $u_{r+1}$. Let $G'$ be the resulting graph. In the pair $(s_i,t_i)$, we replace $s_i$ by $u_{r+1}$ to obtain a new instance $(G',T')$. As $F$ is an induced subgraph of $P_r$,  we have that $G'$ is $H$-free. Hence, by our assumption, we can solve {\sc $k$-Induced Disjoint Paths} on $(G',T')$ in polynomial time. As the total number of branches is polynomial, the total running time is polynomial.
\qed
\end{proof}

\noindent
We need two known results for proving the polynomial part of Theorem~\ref{thm:k-IDP} in Lemma~\ref{l-poly}.

\begin{theorem}[\cite{FKLP12}]\label{t-claw}
For every $k\geq 2$, {\sc $k$-Induced Disjoint Paths} is polynomial-time solvable for claw-free graphs.
\end{theorem}

\begin{theorem}[\cite{Al04}]\label{t-lp}
If a connected chair-free graph $G$ contains an induced claw and an induced path $P$ on at least eight vertices, then $G$ has a vertex adjacent to all vertices of $P$.
\end{theorem}

\begin{lemma}\label{l-poly}
Let $k\geq 2$. For every linear forest~$F$, {\sc $k$-Induced Disjoint Paths} is polynomial-time solvable for ($F+\mbox{chair})$-free graphs.
\end{lemma}

\begin{proof}
By Lemma~\ref{l-sp1}, it remains to consider chair-free graphs.
Let $(G,T)$ be an instance of {\sc Induced Disjoint Paths}, where $G$ is a chair-free graph on $n$ vertices and $T=\{(s_1,t_1),\ldots,(s_k,t_k)\}$ is a set of terminal pairs. 
Let $(P^1,\ldots,P^k)$ be a solution for $(G,T)$ (if it exists). We call a path $P^i$ {\it long} if it has at least eight vertices; else we call it {\it short}.
We first guess which of the paths of a solution for $(G,T)$ will be short. There are $2^k$ options for doing this, which is a constant number as $k$ is a constant. We will consider each of these options one by one.

Suppose we consider the option where $T'\subseteq T$ is the subset of terminal pairs that will be in short solution paths. Let $|T'|=k'\leq k$.
We guess all $O(n^{5k'})=O(n^{5k})$ options of choosing the inner vertices of the solution paths for the terminal pairs in $T'$. We discard an option if two of the guessed solution paths contain an edge between them or if a guessed solution path contains a vertex with a neighbour in some $(s_i,t_i)\notin T'$. Otherwise, we continue as follows.

We first delete all vertices of the guessed solution paths and also their neighbours from $G$. We denote the new instance by $(G,T)$ again and also write $T=\{(s_1,t_1),\ldots, (s_k,t_k)\}$. Assuming our guess was correct, $(G,T)$ only has solutions $(P^1,\ldots,P^k)$ in which each $P^i$ is long. Hence, from $G$, we can safely remove for every $i\in \{1,\ldots,k\}$, every vertex that is adjacent to both $s_i$ and $t_i$.

We now check in polynomial time if there are two terminal $s_i$ and $t_i$ that belong to different connected components of the resulting graph $G'$. If so, then we can discard this branch. Else, we let $(G'_1,T'_1),\ldots, (G'_r,T'_r)$ be the connected components of $G'$, together with  the terminal pairs subsets of $T$ they contain. 

We consider each $(G'_j,T'_j)$ as a separate instance. 
If $G_j'$ has an induced claw, consider a path $P^i$ in a solution. As $P^i$ must be long, Theorem~\ref{t-lp} tells us that $G'_j$ must contain a vertex adjacent to all vertices of $P^i$.
However, by construction, $G'_j$ contains no vertices adjacent to both $s_i$ and $t_i$ which both belong to $P^i$, a contradiction. We now check in polynomial time if $G_j'$ is claw-free. If it is not, then we may discard the branch, as just argued. Otherwise, we apply Theorem~\ref{t-claw} to check in polynomial time if $(G_j',T_j')$ has a solution. If for some $(G_j',T_j')$ no solution exists, then we move to the next branch; otherwise, we return a yes-answer.

As the number of branches is polynomial and processing each branch takes polynomial time, the total running time of our algorithm is polynomial. 
\qed
\end{proof}

\section{NP-Completeness Results}\label{s-np}

In this section we prove the \NP-completeness part of Theorem~\ref{thm:k-IDP} (see Section~\ref{s-mainmain} for details on how we combine the several hardness results proven in this section).

We first prove the \NP-completeness part of Theorem~\ref{thm:k-IDP}.
We base our proof on a hardness result of Leveque et al.~\cite{LLMT09} for $2$-{\sc Induced Cycle}, which is to decide if a graph has an induced cycle containing two pre-specfied vertices $x$ and $y$ (we assume without loss of generality that the induced cycle is a \emph{hole}, meaning it has at least four vertices). Namely, we derive the following relation.
 
\begin{lemma}\label{c-icred}
An instance $(G,x,y)$ of $2$-{\sc Induced Cycle}, where $x$ and $y$ have degree~$2$, can be transformed in polynomial time into an instance of $2$-{\sc Induced Disjoint Paths} on a graph~$G'$. Any vertex that is introduced has degree at most~$3$ and its incident edges can be subdivided an arbitrary number of times.
\end{lemma}

\begin{proof}
Let $x$ and $y$ have neighbours $x_1$, $x_2$ and $y_1$, $y_2$ respectively. We replace $x$ and its incident edges by the following gadget. Create vertices $p_1$, $q_1$, $r_1$, $p_2$, $q_2$, $r_2$, $s_1$, $s_2$. Add edges 
 $s_1p_1, p_1q_1,q_1x_1,
q_1r_1, r_1s_2$ and 
$s_2r_2, r_2q_2, q_2x_2,
p_2s_1, p_2q_2$. Observe that the paths $s_1,p_1,q_1,x_1$ and $s_2,r_2,q_2,x_2$ are mutually induced. Similarly, the paths $s_1,p_2,q_2,x_2$ and $s_2,r_1,q_1,x_1$ are mutually induced. Moreover, these are the only two options that can co-exist, in the sense that a path originating in $s_1$ or $s_2$ that uses only edges of this gadget, does not have $s_1$ or $s_2$ as an internal vertex, and goes to $x_1$ ($x_2$) has to pass through $q_1$ ($q_2$).
In a similar manner, we replace $y$ and its incident edges by vertices $a_1,b_1,c_1,a_2,b_2,c_2,t_1,t_2$ and edges $a_1t_1, a_1b_1, b_1y_1, b_1c_1, c_1t_2$ and $a_2t_1, a_2b_2, b_2y_2, b_2c_2, c_2t_2$. 

We call the resulting graph~$G'$. Then, using the preceding argument, $G$ has a hole containing $x$ and $y$ if and only if $G'$ has mutually induced paths between $s_1$ and $t_1$ and between $s_2$ and $t_2$.
Note that any vertex that is introduced has degree at most~$3$ and any of its incident edges can be subdivided an arbitrary number of times without affecting the correctness of the reduction ($q_1,q_2,b_1,b_2$ remain bottlenecks).
\qed
\end{proof}

\noindent
Our first two results require a single change to the construction of~\cite{LLMT09}. 
The first rectifies a potential issue with the same claim made in~\cite{GPV15} by using Lemma~\ref{c-icred}.

\begin{lemma}\label{l-k14}
$2$-{\sc Induced Disjoint Paths} is \NP-complete for $K_{1,4}$-free graphs.
\end{lemma}

\begin{proof}
Leveque et al.~\cite{LLMT09} prove that $2$-{\sc Induced Cycle} is \NP-complete on graphs of maximum degree~$3$ where the distinguished vertices have degree~$2$. Apply the reduction of Lemma~\ref{c-icred}. The graph has maximum degree~$3$; thus, it is $K_{1,4}$-free.
\qed\end{proof}

\begin{lemma}\label{l-girth}
For every $s\geq 3$ with $s\neq 6$, {\sc $2$-Induced Disjoint Paths} is \NP-complete for $C_s$-free graphs.
\end{lemma}

\begin{proof}
Leveque et al.~\cite{LLMT09} proved that $2$-{\sc Induced Cycle} is \NP-complete on $C_s$-free graphs for every $s\geq 3$ with $s\neq 6$, where the two distinguished vertices have degree~$2$. We apply Lemma~\ref{c-icred} to reduce to $2$-{\sc Induced Disjoint Paths}. Since we can subdivide any number of times the edges incident on the newly created vertices, we can ensure that no induced $C_s$ is created in the final instance.
\qed
\end{proof}

\noindent
Our third result requires a significant overhaul of the construction~in~\cite{LLMT09}. 

\medskip
\noindent
{\bf 3.1. Omitting ``H''-graphs and Six-Vertex Cycles.}
\noindent
Let $H_1$ be the ``H''-graph on six vertices formed by an edge joining the middle vertices of two paths on three vertices. For $\ell\geq 2$, let $H_\ell$ be the graph obtained from $H_1$ by subdividing the crossing edge (which is the edge whose endpoints both have degree~$3$) $\ell-1$ times. 
See Fig.~\ref{fig:H1+H3} for two examples.

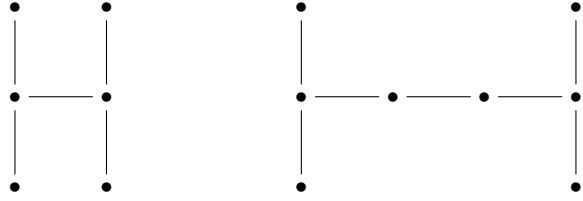
\begin{figure}
\begin{center}
$
\xymatrix{
\bullet  \ar@{-}[d] & \bullet   \ar@{-}[d] \\
\bullet  \ar@{-}[d] \ar@{-}[r] & \bullet   \ar@{-}[d] \\
\bullet   & \bullet  \\
}
$
\hspace{2cm}
$
\xymatrix{
\bullet  \ar@{-}[d] & & & \bullet   \ar@{-}[d] \\
\bullet  \ar@{-}[d] \ar@{-}[r] & \bullet  \ar@{-}[r] & \bullet  \ar@{-}[r] & \bullet   \ar@{-}[d] \\
\bullet  & & & \bullet   \\
}
$
\end{center}
\caption{A drawing of $H_1$ (left) and $H_3$ (right).}
\label{fig:H1+H3}
\end{figure}

We prove that for every $\ell\geq 1$, {\sc $2$-Induced Disjoint Paths} is \NP-complete for $(C_6,H_\ell)$-free graphs.
To this end, we consider the hardness reduction by Leveque et al.~\cite{LLMT09} for $2$-{\sc Induced Cycle} in more detail. 
We very closely follow their notation and the proof of our main Lemma~\ref{l-h} mimics the proof of their Lemma 2.6.
We show how their construction can be modified so that it becomes $H_\ell$-free for any fixed $\ell \geq 1$ and $C_6$-free. 

Let $\phi$ be an instance of {\sc $3$-Satisfiability} consisting of $m$ clauses $C_1,\ldots,C_m$ on $n$ variables $z_1,\ldots,z_n$. For each 
clause 
$C_j$ of the form $y_{3j-2}\vee y_{3j-1} \vee y_{3j}$ then $y_i$, $i \in [3m]$, is a literal from $\{z_1,\ldots,z_n,\overline{z}_1,\ldots,\overline{z}_n\}$. 
Let $\ell\geq 1$ be given. We will construct a graph $G_\phi^\ell$ with two specified vertices~$x$ and $y$ of degree~$2$ so that~$G_\phi^\ell$ has a hole containing $x$ and $y$ if and only if there is a truth assignment satisfying~$\phi$.

\begin{figure}
\[
\xymatrix{
& \alpha^{1+} \ar@/_1.5pc/@{-}[ddddd] \ar@/^/@{-}[dddddr] \ar@{-}[dd] \ar@{-}[ddr] \ar@{-}[r] & \alpha^{1++} \ar@{-}[dd] \ar@{-}[ddl]  \ar@/_/@{-}[dddddl]  \ar@/^1.5pc/@{-}[ddddd]  \ar@{--}[r] & \alpha^{2+} \ar@{-}[r] & \alpha^{3+} \ar@{--}[r]  & \alpha^{4++} \ar@/_1.5pc/@{-}[ddddd] \ar@/^/@{-}[dddddr] \ar@{-}[dd] \ar@{-}[ddr] \ar@{-}[r] & \alpha^{4+} \ar@{-}[dd] \ar@{-}[ddl]  \ar@/_/@{-}[dddddl]  \ar@/^1.5pc/@{-}[ddddd] \ar@{-}[dr] \\
\alpha \ar@{-}[ur] \ar@{-}[dr] & & & & & & &  \alpha' \\
& \alpha^{1-} \ar@{-}[d] \ar@{-}[dr] \ar@{-}[r] & \alpha^{1--} \ar@{-}[d] \ar@{-}[d]  \ar@{-}[dl] \ar@{--}[r] &  \alpha^{2-} \ar@{-}[r] & \alpha^{3-} \ar@{--}[r]  & \alpha^{4--} \ar@{-}[d] \ar@{-}[dr]  \ar@{-}[r] & \alpha^{4-} \ar@{-}[d] \ar@{-}[dl] \ar@{-}[ur] \\
& \beta^{1+} \ar@{-}[dd] \ar@{-}[ddr] \ar@{-}[r] & \beta^{1++} \ar@{-}[dd] \ar@{-}[ddl] \ar@{--}[r] & \beta^{2+} \ar@{-}[r] & \beta^{3+} \ar@{--}[r]  & \beta^{4++} \ar@{-}[r] \ar@{-}[dd] \ar@{-}[ddr]   & \beta^{4+} \ar@{-}[dd] \ar@{-}[ddl] \ar@{-}[dr] \\
\beta \ar@{-}[ur] \ar@{-}[dr] & & & & & & &  \beta' \\
& \beta^{1-} \ar@{-}[r] & \beta^{1--}  \ar@{--}[r] &  \beta^{2-} \ar@{-}[r] & \beta^{3-} \ar@{--}[r]  & \beta^{4--} \ar@{-}[r] & \beta^{4-} \ar@{-}[ur] \\
}
\]
\caption{The literal gadget 
(dashed lines indicate paths of length $\ell$).}
\label{fig:2IDP-literal}
\end{figure}
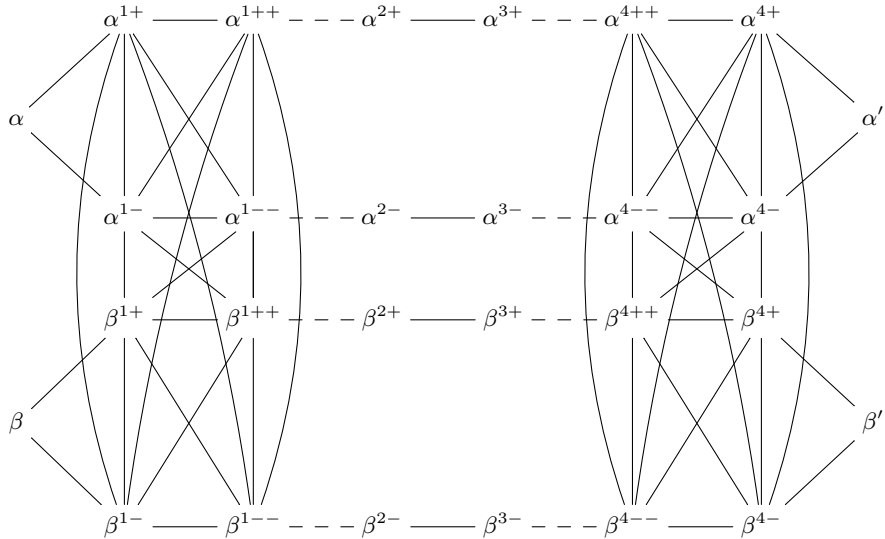

\begin{figure}
\[
\xymatrix{
& \alpha^{1-} \ar@{-}[r] & \alpha^{1--} \ar@{--}[r] &  \alpha^{2-} \ar@{-}[r] & \alpha^{3-} \ar@{--}[r]  & \alpha^{4--} \ar@{-}[r] & \alpha^{4-}  \\
& \beta^{1-}  \ar@{-}[r] & \beta^{1--} \ar@{--}[r] & \beta^{2-} \ar@{-}[r] & \beta^{3-} \ar@{--}[r]  & \beta^{4--} \ar@{-}[r]   & \beta^{4-}  \\
& & & c^{1+} \ar@{-}[u] \ar@{-}[uur] \ar@{-}[ur] \ar@/^1pc/@{-}[uu] & c^{1-} \ar@{-}[uul] \ar@{-}[u] \ar@{-}[ul] \ar@/_1pc/@{-}[uu] & & & \\
}
\]

\[
\xymatrix{
& & c^{1+} & c^{1-} \ar@{--}[dr] & & \\
& c^{12+} \ar@{--}[dr] \ar@{--}[ur] & & & c^{12-} \\
c^{0+}  \ar@{--}[drr] \ar@{--}[ur] & &  c^{2+} & c^{2-} \ar@{--}[ur]  & & c^{0-} \ar@{--}[dll] \ar@{--}[ul]   \\
& & c^{3+} &  c^{3-} & & \\
}
\]
\caption{The clause gadget 
together with its interface with the literal gadget (drawn above). Dashed lines indicate paths of length $\ell$.}
\label{fig:2IDP-clause}
\end{figure}
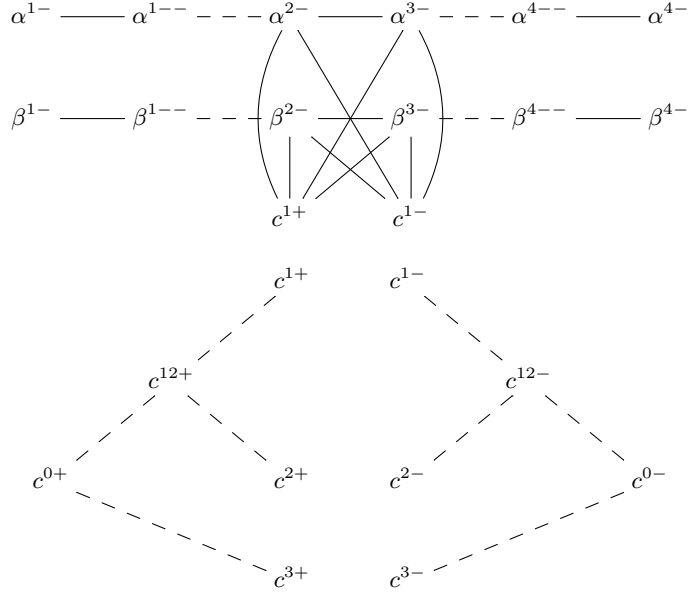

For each literal $y_j$, prepare a graph $G^\ell(y_j)$ as drawn in Fig.~\ref{fig:2IDP-literal} where the corresponding labelled vertices inherit a subscript $j$. Numerous vertices on paths will remain unlabelled. Our literal gadget is more elaborate than that in \cite{LLMT09} as we need to
forbid, as induced subgraphs, the $C_6$ and for any fixed $\ell$, every $H_\ell$.
The idea is that two induced disjoint paths may be drawn through this gadget either crossing the edges $(\alpha^{2+},\alpha^{3+})$ and $(\beta^{2+},\beta^{3+})$; or $(\alpha^{2-},\alpha^{3-})$ and $(\beta^{2-},\beta^{3-})$. All other possibilities are forbidden.

For each clause $C_j$, prepare a graph $G^\ell(C_j)$ as drawn in the bottom of Fig.~\ref{fig:2IDP-clause} where the corresponding labelled vertices inherit a subscript $j$. Numerous vertices on paths will remain unlabelled. Our clause gadget is exactly the same as in \cite{LLMT09} except we replaced the edges by paths.
The idea is that a path may be drawn through this gadget in precisely one of three ways selecting the literal that is true.

For each variable $z_i$, prepare a graph $G^\ell(z_i)$ as in Fig.~\ref{fig:2IDP-variable} consisting of two internally disjoint paths $P^+_i$ (top) and $P^-_i$ (bottom). The idea in Fig.~\ref{fig:2IDP-variable} is that full edges and dashed edges alternate on this diagram and the length is enough for $m$ full edges. The end points of the full edges are labelled $(p_{i,1}^+,p_{i,1}^{++})$, \ldots, $(p_{i,2m}^+,p_{i,2m}^{++})$ on the top; and $(p_{i,1}^-,p_{i,1}^{--})$, \ldots, $(p_{i,2m}^-,p_{i,2m}^{--})$ on the bottom. Our variable gadget is exactly the same as in \cite{LLMT09} except we lengthened some paths.
The idea is that a path may be drawn through this gadget in precisely one of two ways selecting whether the variable is evaluated true or false.

The final graph $G_\phi^\ell$ is constructed in a manner similar to Leveque et al.~\cite{LLMT09} from the disjoint union of all the graphs $G^\ell(y_j)$ (literals), $G^\ell(C_j)$ (clauses) and $G^\ell(x_i)$ (variables) with the modifications as below.
We indicate specifically where the modifications go beyond the construction of Lemma 2.6 in \cite{LLMT09}. The top of Fig.~\ref{fig:2IDP-clause} shows how a clause gadget interacts with a literal gadget. Note that a variable gadget interacts with a clause gadget in a similar way.

\begin{enumerate}
\item In \cite{LLMT09}, for $j=1,\ldots,3m-1$, they added the edges $\alpha'_j\alpha_{j+1}$ and $\beta'_j\beta_{j+1}$. We will instead add paths of length $\ell$ in place of these edges.
\item In \cite{LLMT09}, for $j=1,\ldots,m-1$, they added the edges $c^{0-}_jc^{0+}_{j+1}$.  We will instead add paths of length $\ell$ in place of these edges.
\item In \cite{LLMT09}, for $i=1,\ldots,m-1$, they add the edges $d^{-}_id^{+}_{i+1}$. We will instead add paths of of length $\ell$  in place of these edges.
\item For $i=1,\ldots,n$, let $y_{n_1},\ldots y_{n_{z^-_i}}$ be the occurrences of $\overline{z}_i$ over all literals.~We have slightly different vertex names from \cite{LLMT09}. For $j=1,\ldots,z^-_i$, delete the edge $p^+_{i,j}p^{++}_{i,j}$ and add the four edges $p^{+}_{i,j}\alpha^{2+}_{n_j}$, $p^{+}_{i,j}\beta^{2+}_{n_j}$, $p^{++}_{i,j}\alpha^{3+}_{n_j}$, $p^{++}_{i,j}\beta^{3+}_{n_j}$. 
Additionally to these edges, which were in \cite{LLMT09}, we also add: $p^{+}_{i,j}\alpha^{3+}_{n_j}$, $p^{+}_{i,j}\beta^{3+}_{n_j}$, $p^{++}_{i,j}\alpha^{2+}_{n_j}$, $p^{++}_{i,j}\beta^{2+}_{n_j}$.
\item For $i=1,\ldots,n$, let $y_{n_1},\ldots y_{n_{z^+_i}}$ be the occurrences of $z_i$ over all literals. We have slightly different vertex names from \cite{LLMT09}. For $j=1,\ldots,z^+_i$, delete the edge $p^-_{i,j}p^{--}_{i,j}$ and add the four edges $p^{-}_{i,j}\alpha^{2+}_{n_j}$, $p^{-}_{i,j}\beta^{2+}_{n_j}$, $p^{--}_{i,j}\alpha^{3+}_{n_j}$, $p^{--}_{i,j}\beta^{3+}_{n_j}$. 
Additionally to these edges, which were in \cite{LLMT09}, we also add: $p^{-}_{i,j}\alpha^{3+}_{n_j}$, $p^{-}_{i,j}\beta^{3+}_{n_j}$, $p^{--}_{i,j}\alpha^{2+}_{n_j}$, $p^{--}_{i,j}\beta^{2+}_{n_j}$.
\item For $i=1,\ldots,m$ and $j=1,2,3$, add the edges $\alpha^{2-}_{3(i-1)+j} c^{j+}_i$, $\alpha^{3-}_{3(i-1)+j} c^{j-}_i$, $\beta^{2-}_{3(i-1)+j} c^{j+}_i$, $\beta^{3-}_{3(i-1)+j} c^{j-}_i$. 
Additionally to these edges, which were in \cite{LLMT09}, we also add: $\alpha^{3-}_{3(i-1)+j} c^{j+}_i$, $\alpha^{2-}_{3(i-1)+j} c^{j-}_i$, $\beta^{3-}_{3(i-1)+j} c^{j+}_i$, $\beta^{2-}_{3(i-1)+j} c^{j-}_i$.
\item In \cite{LLMT09}, they add the edges $\alpha'_{3m}d^+_1$ and $\beta'_{3m}c^{0+}_1$. Instead we will add a path of length $\ell$.
\item Add the vertex $x$. In \cite{LLMT09}, they  add the edges $x\alpha_1$ and $x\beta_1$. Instead we will add paths of length $\ell$.
\item Add the vertex $y$. In \cite{LLMT09}, they add the edges $yc^{0-}_m$ and $yd^{-}_n$. Instead we will add paths of length $\ell$.
\end{enumerate}

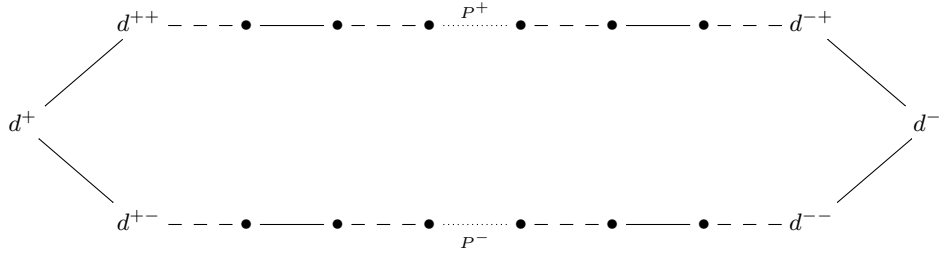
\begin{figure}[tb]
\[
\xymatrix{
& d^{++} \ar@{-}[dl] \ar@{--}[r] & \bullet \ar@{-}[r] & \bullet \ar@{--}[r] & \bullet \ar@{.}[r]^{P^+} & \bullet \ar@{--}[r] & \bullet  \ar@{-}[r] & \bullet \ar@{--}[r] & d^{-+} \ar@{-}[dr] & \\
d^{+} & & & & & & &  & & d^{-} \\
& d^{+-} \ar@{-}[ul] \ar@{--}[r] & \bullet \ar@{-}[r] & \bullet \ar@{--}[r] & \bullet \ar@{.}[r]_{P^-} & \bullet \ar@{--}[r] & \bullet  \ar@{-}[r] & \bullet \ar@{--}[r] & d^{--}  \ar@{-}[ur] & \\
}
\]
\caption{The variable gadget.
 Dashed lines indicate paths of length $\ell$. Dotted lines indicate a continuation of the gadget.}
\label{fig:2IDP-variable}
\end{figure}

\begin{claimm} \label{c-our-leveque-2.6}
$\phi$ is satisfied by a truth assignment if and only if $G^\ell_\phi$ contains a hole passing through $x$ and $y$.
\end{claimm}
\begin{proof}
The idea is that any hole emanating from $x$ and moving rightwards towards $y$ (see Fig.~\ref{fig:2IDP-literal}) must traverse the literal gadgets in precisely one of two ways (upper path on top and bottom; or bottom path on the top and bottom). Now, subsequently the paths building the hole may return to these literal gadgets but they can never leave them as each $\alpha,\beta,\alpha',\beta'$ are already traversed. 
Subsequently, the paths do indeed not return to the literal gadgets to ensure their consistent evaluation with the variables and that one in each clause is true.

More formally, first assume that $\phi$ is satisfied by a truth assignment $\xi \in \{0,1\}^n$. We pick a set of vertices that induce a hole containing $x$ and~$y$.
\begin{enumerate}
\item Pick vertices $x$ and $y$.
\item For $i=1,\ldots,3m$, pick  $\alpha_i,\alpha'_i,\beta_i,\beta'_i$.
\item For $i=1,\ldots,3m$, if $y_i$ is satisfied by $\xi$, then pick $\alpha_i^{1+}$, $\alpha_i^{1++}$, $\alpha_i^{2+}$, $\alpha_i^{3+}$,$\alpha_i^{4++},$ $\alpha_i^{4+}$ and any vertices on a direct path between these. Else, pick  $\alpha_i^{1-},\alpha_i^{1--},$ $\alpha_i^{2-},\alpha_i^{3-},\alpha_i^{4--},\alpha_i^{4-}$ and any vertices on a direct path between these. 
\item For $i=1,\ldots,n$, if $\xi(i)=1$, then pick all vertices of $P^+_i$ and all the neighbours of the vertices in $P^+_i$ of the form $\alpha_k^{2+}$ (or one could choose $\alpha_k^{3+}$, but only one among the two) for any $k$. Additionally pick any vertices on a direct path between these.
\item For $i=1,\ldots,n$, if $\xi(i)=0$, then pick all the vertices of $P^-_i$ and all the neighbours of the vertices in $P^-_i$ of the form $\alpha_k^{2+}$ 
(or one could choose $\alpha_k^{3+}$, but only one among the two) for any $k$. Additionally pick any vertices on a direct path between these.
\item For $i=1,\ldots,m$, pick the vertices $c_i^{0+}$ and $c_i^{0-}$. Choose any $j \in \{3i-2,3i-1,3i\}$ such that $\xi$ satisfies $y_j$. Pick vertices $\alpha^{2-}_j$ and $\alpha^{3-}_j$.
\begin{itemize}
\item If $j=3i-2$, then pick $c_i^{12+},c_i^{1+},c_i^{1-},c_i^{12-}$ as well as all vertices on a path between: $c^{0+}$ and $c_i^{12+}$; $c_i^{12+}$ and $c_i^{1+}$; $c^{0-}$ and $c_i^{12-}$; $c_i^{12-}$ and $c_i^{1-}$. 
\item If $j=3i-1$, then pick $c_i^{12+},c_i^{2+},c_i^{2-},c_i^{12-}$ as well as all vertices on a path between: $c^{0+}$ and $c_i^{12+}$; $c_i^{12+}$ and $c_i^{2+}$; $c^{0-}$ and $c_i^{12-}$; $c_i^{12-}$ and $c_i^{2-}$. 
\item If $j=3i$, then pick $c_i^{3+},c_i^{3-}$ as well as all vertices on a path between: $c^{0+}$ and $c_i^{3+}$; $c_i^{0-}$ and $c_i^{3-}$.
\end{itemize}
\end{enumerate}
It suffices to show that the chosen vertices induce a hole in $G_\phi^\ell$ containing $x$ and $y$. The only potential problem is that for some $k$, one of the vertices $\alpha_k^{2+},\alpha_k^{3+},\alpha_k^{2-},\alpha_k^{3-}$ was chosen more than once. If $\alpha_k^{2+}$ and $\alpha_k^{3+}$ were picked in Step 3, then $y_k$ is satisfied by~$\xi$. Therefore, $\alpha_k^{2+}$ and $\alpha_k^{3+}$ were not chosen in Step 4 or Step 5. Similarly, if $\alpha_k^{2-}$ and $\alpha_k^{3-}$ were picked in Step 6, then $y_k$ is satisfied by~$\xi$. Therefore, $\alpha_k^{2-}$ and $\alpha_k^{3-}$ 
were not chosen in Step 3. Thus, the chosen vertices 
induce a hole in $G_\phi^\ell$ containing $x$ and $y$. 

Now assume that $G_\phi^\ell$ has a hole including $x$ and $y$. The hole must contain $\alpha_1$ and $\beta_1$ since they are the only neighbours of $x$. Next, either both $\alpha_1^{1+}$ and $\beta_1^{1+}$ are in the hole or both $\alpha_1^{1-}$ and $\beta_1^{1-}$. W.l.o.g., let $\alpha_1^{1+}$ and $\beta_1^{1+}$ be in the hole (the same reasoning will apply in the other case). Since $\alpha_1^{1-}$, $\beta_1^{1-}$, $\alpha_1^{1--}$, $\beta_1^{1--}$ are all neighbours of two vertices in the hole, they cannot themselves be in the hole. Thus, $\alpha_1^{2^+}$, $\beta_1^{2+}$, and the paths that lead to them, must be in the hole. Since $\alpha_1^{2^+}$, $\beta_1^{2+}$ have the same neighbourhood outside of $G(y_1)$ it follows that $\alpha_1^{3^+}$, $\beta_1^{3+}$ must be in the hole. Indeed, so must also $\alpha_1^{4++}$, $\beta_1^{4++}$, $\alpha_1^{4^+}$, $\beta_1^{4+}$ and the path in between. 
Note that $\alpha_1^{4-}$, $\beta_1^{4-}$ are not in the hole, as they are adjacent to both $\alpha_1^{4++}$ and $\beta_1^{4++}$. So it must contain instead $\alpha'_1$, $\beta'_1$, $\alpha_2$, $\beta_2$.
By induction, we see for $i \in [3m]$ that the hole must contain $\alpha_i$, $\beta_i$, $\alpha'_i$, $\beta'_i$. Also, for each $i$, the hole must contain $\alpha_i^{1+},\alpha_i^{1++},\ldots,\alpha_i^{2+},\alpha_i^{3+},\ldots,\alpha_i^{4++},\alpha_i^{4+}$ or $\alpha_i^{1-},\alpha_i^{1--},\ldots,\alpha_i^{2-},\alpha_i^{3-},$ $\ldots,\alpha_i^{4--},\alpha_i^{4-}$.
Hence, the hole contains $d_1^+$ and~$c_1^{0+}$. 

By symmetry we may assume the hole contains $d_1^{++}$, and the path to $p_{1,1}^+$, and $\alpha_k^{2+}$ for some $k$. As $\alpha_k^{1++}$ is adjacent to two vertices in the hole, the hole must contain one of $\alpha_k^{2+}$ and $\alpha_k^{3+}$. Similarly, the hole cannot proceed on a path to $\alpha_k^{4++}$, so it must contain $p^+_{1,2}$ and $p^{++}_{1,2}$. By induction, 
we see 
that the hole contains $p^+_{1,i},p^{++}_{1,i}$, for $i \in [n]$, and $d^-_1$. If the hole contains $d^{--}_1$, then the hole must contain $p^-_{1,i},p^{--}_{1,i}$, for $i \in [n]$, and eventually $d^{+-}_1$, a contradiction. Thus, the hole must contain $d_2^+$. By induction, for $i \in [n]$, 
we see that 
the hole contains all the vertices of the path $P_i^+$ or $P_i^-$ and, by symmetry, we assume that the hole contains neighbours of the vertices in $P^+_i$ or $P^-_i$, one among $\alpha_k^{2+}$ and $\alpha_k^{3+}$, for each $k$.

Similarly, for $i \in [m]$, it follows that the hole must contain $c_i^{0+}$ and~$c_i^{0-}$. The hole also contains one of the following:
\begin{itemize}
\item $c_i^{12+}$, $c_i^{1+}$, $c_i^{1-}$, $c_i^{12-}$, and the paths between, and either one of $\alpha_j^{2-}$, $\alpha_j^{3-}$; or one of $\beta_j^{2-}$, $\beta_j^{3-}$.
\item $c_i^{12+}$, $c_i^{2+}$, $c_i^{2-}$, $c_i^{12-}$, and the paths between, and either one of $\alpha_j^{2-}$, $\alpha_j^{3-}$; or one of $\beta_j^{2-}$, $\beta_j^{3-}$.
\item $c_i^{3+}$, and the path between, and either one of $\alpha_j^{2-}$, $\alpha_j^{3-}$; or one of $\beta_j^{2-}$, $\beta_j^{3-}$.
\end{itemize}
We now recover the satisfying assignment $\xi$. For $i \in [n]$, 
set $\xi(i)=1$ if the vertices of $P_i^+$ are in the hole; 
otherwise set $\xi(i)=0$.
By construction,
 at least one literal in every clause is satisfied by $\xi$, so  indeed $\xi$ is a satisfying assignment.
\qed\end{proof}

\begin{claimm} \label{c-hc6free}
The graph $G^\ell_\phi$ is $C_6$-free and $H_i$-free for every $i \in [\ell]$.
\end{claimm}

\begin{proof}
Owing to the length of the $\ell$ paths 
that populate our construction and are drawn as dashed edges in our figures, we need only verify the omission of the relevant graphs on the connected components of the graph $G_\phi$ after the removal of these $\ell$ paths that are dashed edges. That would suffice for $C_6$, but $H_i$ has a pendant edge, so for these we must leave a pendant edge from the corresponding connected component at the extremities of an instance of these $\ell$ paths that are drawn as dashed edges. 
In this fashion, we only need to check for omission of the given graphs in the non-trivial cases drawn in Fig.~\ref{fig:2IDP-omission-cases}. 
It can be readily observed that these graphs are $P_7$ free (but they are not $P_6$-free). Hence, we need not test beyond $H_3$.
This task was accomplished by a program testing subgraph isomorphism whose code we provide a link to.\footnote{See \texttt{https://github.com/barnabymartin/InducedSubgraph}.}

We will give an explicit argument for the case of $C_6$-freeness, which is simpler as $C_6$ has numerous symmetries (a transitive automorphism group). 
Let us begin with the graph depicted on the left-hand side of Fig.~\ref{fig:2IDP-omission-cases}. This graph has an automorphism that swaps $\alpha$ and $\beta$ at the same time as $+$ and $-$. It also has an automorphism that only swaps $+$ and $-$. 
Any subgraph that induces a $C_6$ cannot contain any of the unlabelled vertices, nor $\alpha$ nor $\beta$. This leaves eight vertices that may be involved. We will consider the case where the $C_6$ contains $\alpha^{1+}$. Owing to the two automorphisms we have described, this argument would equally apply to $\alpha^{1-}$ and $\beta^{1-}$. But any $C_6$ must involve one of these vertices as there were only eight to choose from. Thus, when we have considered this case, our work is done:

\medskip
\noindent
{\bf Subcase A.} The $C_6$ contains $\alpha^{1+}$ and $\alpha^{1++}$. All other neighbours of $\alpha^{1+}$ (except $\alpha$) are adjacent to $\alpha^{1++}$. No $C_6$ can be formed here.

\noindent
{\bf Subcase B.} The $C_6$ contains $\alpha^{1+}$ and $\alpha^{1--}$. Any $C_6$ involving a path $\alpha^{1--}$ to $\alpha^{1+}$ must next go to $\beta^{1-}$. 
We cannot continue this cycle.

\noindent
{\bf  Subcase C.} The $C_6$ contains $\alpha^{1+}$ and $\beta^{1--}$. Any $C_6$ involving a path $\beta^{1--}$ to $\alpha^{1+}$ must next go to $\alpha^{1-}$ or $\alpha^{1--}$. We cannot continue this cycle.

\noindent
{\bf Subcase D.} The $C_6$ contains $\alpha^{1+}$ and $\beta^{1-}$. Now, $\beta^{1-}$ can have as the next in the cycle either of $\beta^{1+}$ or $\beta^{1++}$.
We cannot continue this cycle.

\noindent
{\bf Subcase E.} The $C_6$ contains $\alpha^{1+}$ and $\alpha^{1-}$. Now, $\alpha^{1-}$ can have as the next in the cycle either of $\beta^{1+}$ or $\beta^{1++}$.
We cannot continue this cycle.

Now we consider the graph depicted on the right-hand side of Fig.~\ref{fig:2IDP-omission-cases}. Any $C_6$ cannot contain any of the unlabelled vertices. It follows that it must use all six remaining vertices. But this induced graph has a triangle, so we are finished.
\qed\end{proof}
We note that the construction in \cite{LLMT09} omits all cycles other than $C_6$, and they note specifically this lacuna, which we have remedied.

We now prove our result.

\begin{figure}[tbp]
$
\xymatrix{
& & \alpha^{1+} \ar@/_1.5pc/@{-}[ddddd] \ar@/^/@{-}[dddddr] \ar@{-}[dd] \ar@{-}[ddr] \ar@{-}[r] & \alpha^{1++} \ar@{-}[dd] \ar@{-}[ddl]  \ar@/_/@{-}[dddddl]  \ar@/^1.5pc/@{-}[ddddd]  \ar@{-}[r] & \bullet \\
\bullet \ar@{-}[r]  & \alpha \ar@{-}[ur] \ar@{-}[dr] & & & \\
& & \alpha^{1-} \ar@{-}[d] \ar@{-}[dr] \ar@{-}[r] & \alpha^{1--} \ar@{-}[d] \ar@{-}[d]  \ar@{-}[dl] \ar@{-}[r] &  \bullet \\
& & \beta^{1+} \ar@{-}[dd] \ar@{-}[ddr] \ar@{-}[r] & \beta^{1++} \ar@{-}[dd] \ar@{-}[ddl] \ar@{-}[r] & \bullet \\
\bullet \ar@{-}[r] & \beta \ar@{-}[ur] \ar@{-}[dr] & & &  \\
& & \beta^{1-} \ar@{-}[r] & \beta^{1--}  \ar@{-}[r] &  \bullet \\
}
$
\ \ \
$
\xymatrix{
\bullet  \ar@{-}[r] &  \alpha^{2-} \ar@{-}[r] & \alpha^{3-} \ar@{-}[r] & \bullet  \\
\bullet \ar@{-}[r] & \beta^{2-} \ar@{-}[r] & \beta^{3-} \ar@{-}[r]  & \bullet  \\
& c^{1+} \ar@{-}[u] \ar@{-}[uur] \ar@{-}[ur] \ar@/^1pc/@{-}[uu] & c^{1-} \ar@{-}[uul] \ar@{-}[u] \ar@{-}[ul] \ar@/_1pc/@{-}[uu] &  \\
\bullet \ar@{-}[ur] & & & \bullet \ar@{-}[ul] \\
}
$
\caption{Cases that need to be checked for omission of the graphs $C_6$ and $H_i$ $(1\leq i\leq \ell)$.}
\label{fig:2IDP-omission-cases}
\end{figure}
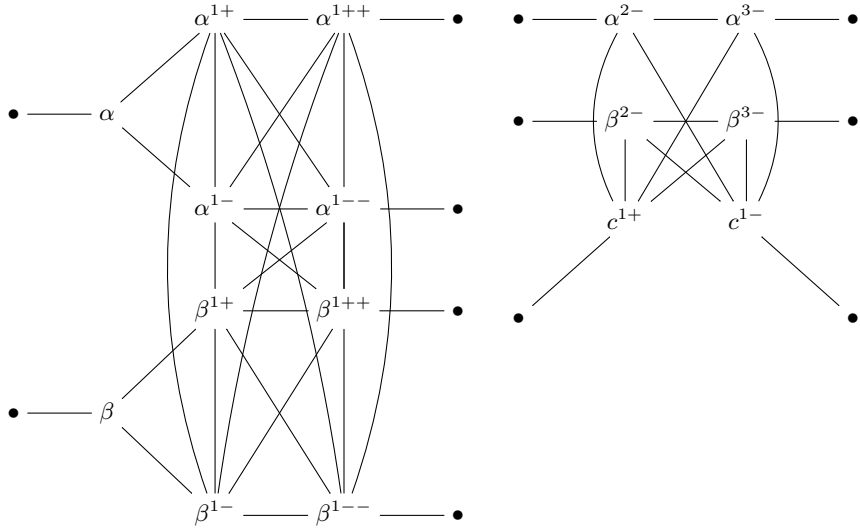

\begin{lemma}\label{l-h}
For every integer $\ell\geq 1$, {\sc $2$-Induced Disjoint Paths} is \NP-complete for $(C_6,H_\ell)$-free graphs.
\end{lemma}

\begin{proof}
We give a reduction from an instance $\phi$ of {\sc $3$-Satisfiability}. First, we construct $G^\ell_\phi$. By Claim~\ref{c-our-leveque-2.6}, $G^\ell_\phi$ has a hole through $x$ and $y$ if and only if $\phi$ is satisfiable. Moreover, $G^\ell_\phi$ is $(C_6,H_\ell)$-free by Claim~\ref{c-hc6free}. We now apply the reduction of Lemma~\ref{c-icred}. As we can subdivide any number of times the edges incident on the newly created vertices, the resulting graph is still $(C_6,H_\ell)$-free.
\qed
\end{proof}

\section{The Proof of Theorem~\ref{thm:k-IDP}}\label{s-mainmain}

We now use the results from the previous two sections to prove our main theorem. 

\medskip
\noindent
{\bf Theorem~\ref{thm:k-IDP} (restated).}
{\it Let $k\geq 2$. For a graph $H$, {\sc $k$-Induced Disjoint Paths} is polynomial-time solvable if $H$ is a subgraph of the disjoint union of a linear forest and a chair, and it is \NP-complete if $H$ is not in ${\cal S}$.}

\begin{proof}
If $H$ has a cycle $C_s$, apply Lemma~\ref{l-girth} for $s\neq 6$ or Lemma~\ref{l-h} for $s=6$. Then we may assume $H$ is a forest. If $H$ has a vertex of degree at least~$4$, then every $K_{1,4}$-free graph is $H$-free, so apply Lemma~\ref{l-k14}. Suppose $H$ has maximum degree at most~$3$. If $H$ has a connected component with at least two vertices of degree~$3$, then $H$ has an induced $H_\ell$, so apply Lemma~\ref{l-h} again. Else, $H$ is in ${\cal S}$. 
If $H$ is a subgraph of the disjoint union of a linear forest and a chair,  apply Lemma~\ref{l-poly}. \qed
\end{proof}

\section{Parameterized Complexity}\label{s-pc}

We prove that {\sc $2$-Induced Disjoint Paths} is $\W[1]$-hard for $P_r$-free graphs when parameterized by $r$. In order to do this we adapt a reduction of Haas and Hoffmann~\cite{HH06}. They prove that finding an induced path through three specified vertices is $\W[1]$-hard when parameterized by the length of the path. The graph of their construction can potentially contain arbitrarily long induced paths, but we propose a modification which guarantees that the length of any induced path in the construction is bounded. Below, we discuss the original construction and our modification; the proof of correctness is almost exactly the same and only sketched here.

\begin{theorem}\label{t-w1}
{\sc $2$-Induced Disjoint Paths} is $\W[1]$-hard on $P_r$-free graphs, parameterized by $r$.
\end{theorem}

\begin{proof}
The reduction is from {\sc Independent Set}, which is known to be $\W[1]$-hard when parameterized by the size of the solution~\cite{DF95}. Let $G=(V,E)$ be an instance of this problem, where $V=\{v_1,\ldots,v_n\}$, and let $k$ be the parameter. The main ingredient of the construction by Haas and Hoffmann~\cite{HH06} is a set of $k$ vertex choice diamonds. The $i$-th diamond consists of $n$ vertices $v_1^i,\ldots,v_n^i$, and two vertices $s^i$ and $t^i$ which are adjacent to all $v_1^i,\ldots,v_n^i$. Our modification is to make $v_1^i,\ldots,v_n^i$ a clique, instead of an independent set as in the original construction. Identify $s^{i-1}$ and $t^i$ for each $2 \leq i \leq k$ and call the resulting graph $G_{VC}$.

Now create two copies of $G_{VC}$. Denote the vertices in the first copy by $s^i$, $v^i_j$, and $t^i$ and in the second copy respectively by $\sigma^i$, $\varphi^i_j$, and $\tau^i$. Add an edge between $t^k$ and $\tau^k$ and subdivide it once (the latter is a minor modification with respect to the original construction). Call the resulting graph $G''$.

Note that $G''$ is a union of $2k+2$ cliques, one of which has size~$1$. In the remainder, we will not add more vertices, only edges. Since any induced path can contain at most two vertices of any clique, the graph is and will remain $P_r$-free for $r \leq 4k+4$.

From $G''$, construct the graph $G'$ by adding the following edges (again, following Haas and Hoffmann~\cite{HH06}):
\begin{itemize}
\item add a \emph{consistency edge} between $v^i_j$ and $\varphi^i_\ell$, for all $1 \leq i \leq k$ and all $1 \leq j,\ell,n$ with $j \not= \ell$ (thus $k n (n-1)$ consistency edges are added in total);
\item add \emph{independence edges} between $\{v^i_p,\varphi^i_p\}$ and $\{v^j_q,\varphi^j_q\}$ for each edge $\{v_p,v_q\} \in E$ and for all $1 \leq i,j \leq k$ with $i\not=j$ (thus $4k(k-1) \cdot |E|$ independence edges are added in total);
\item add \emph{set edges} between $\{v^i_\ell,\varphi^i_\ell\}$ and $\{v^j_\ell,\varphi^j_\ell\}$ for all $1 \leq i,j \leq k$ with $i \not= j$ and all $1 \leq \ell \leq n$ (thus $4k(k-1)n$ set edges are added in total).
\end{itemize}
This completes the construction. Let the vertex pairs for the instance be $(s^1,t^k)$ and $(\sigma^1,\tau^k)$.

To show correctness, we essentially repeat the arguments of Haas and Hoffmann~\cite[Lemma~7, Theorem~8]{HH06}. We sketch the argument below. Let $P^1, P^2$ be mutually induced disjoint $(s^1,t^k)$- and $(\sigma^1,\tau^k)$-paths respectively. By shortcutting if necessary, we may assume that $P^1$ and $P^2$ are chordless. By construction, $P^1$ must contain one of $\{v_1^1,\ldots,v_n^1\}$, say $v_j^1$, and $P^2$ must contain one of $\{\varphi_1^1,\ldots,\varphi_n^1\}$, say $\varphi_\ell^1$. By the consistency edges, $j=\ell$. Also note that since $v_1^1,\ldots,v_n^1$ and $\varphi_1^1,\ldots,\varphi_n^1$ both induce a clique, $P^1$ cannot follow a consistency edge from $v_j^1$ and $P^2$ cannot follow a consistency edge from $\varphi_j^1$. Also note that the independence and set edges incident on $v_j^1$ and $\varphi_j^1$ lead to exactly the same vertices, and thus these vertices cannot be part of $P^1$ nor $P^2$. Hence, $P^1$ must continue to $t^1=s^2$ and $P^2$ must continue to $\tau^1=\sigma^2$. By repeating the same argument, we can show that $P^1$ and $P^2$ only use edges of the diamonds, and none of the consistency, independence, and set edges. In particular, $P^1$ and $P^2$ use vertices $s^1,v^1_{\gamma_1},t^1,\ldots,v^k_{\gamma_k},t^k$ and $\sigma^1,\varphi^1_{\delta_1},\tau^1,\ldots,\varphi^k_{\delta_k},\tau^k$ respectively. The consistency, independence, and set edges ensure respectively that $\gamma_i = \delta_i$ for $1 \leq i \leq k$, that $v_{\gamma_1},\ldots,v_{\gamma_k}$ form an independent set, and that $v_{\gamma_1},\ldots,v_{\gamma_k}$ is a set of size $k$.

For the converse, it is easy to see that any independent set $I=\{v_{\gamma_1},\ldots,v_{\gamma_k}\}$ can be transformed in a solution to the instance. The vertices $s^1,v_{\gamma_1}^1,t^1,\ldots,v_{\gamma_k}^k,t^k$ form an $(s^1,t^k)$-path that is mutually disjoint from the $(\sigma^1,\tau^k)$-path formed by  $\sigma^1,\varphi_{\gamma_1}^1,\tau^{1},\ldots,\varphi_{\gamma_k}^k,\tau^k$. \qed
\end{proof}

\section{Conclusions}\label{s-con}

We showed new tractable and hard results for {\sc $k$-Induced Disjoint Paths} for $H$-free graphs and extended a number of known results in this way.
The open cases all involve graphs $H$ that are not the disjoint union of some linear forest and the chair but that do belong to the family ${\cal S}$; we recall that ${\cal S}$ consists of all graphs, every 
connected component of which is a path $P_r$ or a subdivided claw~$S_{h,i,j}$. 

Due to the above, {\sc $k$-Induced Disjoint Paths} belongs to a set of several other problems whose complexity is open for $S_{h,i,j}$-free graphs for many $h,i,j$. The best-known problem of this set of problems is {\sc Independent Set}, which is to decide if a graph has an independent set of size at least $p$ for some integer~$p$.
Alekseev~\cite{Al82} proved that if a graph~$H$ is not in ${\cal S}$, then {\sc Independent Set} is \NP-complete for $H$-free graphs. If $H\in {\cal S}$, only a restricted number of cases are known to be polynomial-time solvable for {\sc Independent Set}. Another example is $3$-{\sc Colouring} for $H$-free graphs of bounded diameter (see~\cite{MPS19}). We do not know any graph $S_{h,i,j}$ and integer~$d$ such that $3$-{\sc Colouring} for $S_{h,i,j}$-free graphs of diameter~$d$ is \NP-complete (and only a small number of polynomial cases exist). Hence, in order to make further progress on  {\sc $k$-Induced Disjoint Paths} and these other problems we must better understand the structure of $S_{h,i,j}$-free graphs. 

It would also be interesting to consider the parameterized complexity of the problem in more detail. First recall that Golovach et al.~\cite{GPV15} showed that {\sc Induced Disjoint Paths} is  \FPT\ for claw-free graphs when parameterized by $k$. The class of claw-free graphs is properly contained in the class of chair-free graphs. From Theorem~\ref{thm:k-IDP} we know that {\sc Induced Disjoint Paths} is \XP\ for chair-free graphs when parameterized by $k$. Is the problem even \FPT\ for chair-free graphs when parameterized by $k$?

Moreover, in Theorem~\ref{t-w1} we proved that even {\sc $2$-Induced Disjoint Paths} is $\W[1]$-hard on $P_r$-free graphs when parameterized by $r$. But what is the parameterized complexity of {\sc Induced Disjoint Paths} for $P_r$-free graphs if $r$ is a constant but $k$ is the parameter? So far, we only know that the problem is in \XP\ by Theorem~\ref{thm:k-IDP} and that we may assume that $r\geq 7$ due to Theorem~\ref{thm:IDPnew}. 
Again we may draw a parallel to the situation for {\sc Independent Set} on $P_r$-free graphs. This problem can be solved in polynomial time for $r < 7$~\cite{GKPP22} and is trivially in \XP\ when parameterized by the size of the solution. However, it is open whether if it is FPT for constant $r \geq 7$ when parameterized by the size of the solution 
(see also~\cite{BBCTW20}). 
Progress on this problem would help to advance our understanding of {\sc $k$-Induced Disjoint Paths} on $P_r$-free graphs 
(see~\cite{MPSV} for a close relationship between both problems when the input is restricted to $P_r$-free graphs for some integer~$r\geq 1$).

Finally, in some previous works, a slightly more general definition is used (see also Section~\ref{s-intro}). 
Given a graph $G$, \emph{vertex-disjoint} paths $P^1,\ldots, P^k$, for some integer $k\geq 1$ are {\it flexibly mutually induced paths} of $G$ if there is no edge between two vertices from different $P^i$ and $P^j$ except possibly between the endpoints of the paths. If $k$ is in the input,  the complexity of the corresponding decision problem and ours is most likely different for $P_r$-free graphs. Namely, {\sc Flexibly Induced Disjoint Paths} is \NP-complete for  $P_{14}$-free graphs~\cite{MPSV}, whilst {\sc Induced Disjoint Path} is quasipolynomial-time solvable for $P_{14}$-free graphs by Theorem~\ref{thm:IDPnew}. However, it is readily seen that all polynomial-time results in Theorem~\ref{thm:k-IDP}  (so, for fixed $k\geq 2$) also hold for {\sc Flexible $k$-Induced Disjoint Paths}. This is even in the case if we also allow that two different paths $P^i$ and $P^j$ share a terminal (this even more general variant has been considered in the literature as well).

\bibliographystyle{abbrv}

\end{document}